\documentclass[11pt]{article}
\usepackage{amsfonts}
\usepackage{latexsym,amsfonts,amssymb,amsmath,amsthm}
\usepackage{graphicx}

\usepackage{latexsym,amsfonts,amssymb,amsmath,amsthm}
\usepackage{graphicx}

\newtheorem{theorem}{Theorem}[section]
\newtheorem{definition}[theorem]{Definition}
\newtheorem{lemma}[theorem]{Lemma}

\newtheorem{remark}[theorem]{Remark}
\newtheorem{prop}[theorem]{Proposition}

\numberwithin{equation}{section}

\begin{document}
\begin{center}

\textbf{\Large Global well posedness and scattering for the }

 \vskip 0.2cm

\textbf{\Large defocusing, cubic  NLS in $\mathbb{R}^3$}

\vskip 1cm

\begin{minipage}[c]{13cm}

\noindent \textbf{Abstract}: We prove global well-posedness and
scattering for the defocusing, cubic NLS on $\mathbb{R}^3$ with
initial data in $H^s(\mathbb{R}^3)$ for $s>49/74$. The proof
combines the ideas of resonance decomposition in \cite{CKSTT4} and
linear-nonlinear decomposition in \cite{ben1}\cite{roy} together
with the idea of large time iteration.

\end{minipage}
\end{center}

\section{Introduction}

\noindent Consider the defocusing cubic NLS in 3D

\begin{equation}\label{equ}
\begin{aligned}
\begin{cases}
iu_t+\Delta u=|u|^2u,~~~~(t,x)\in\mathbb{R}_+\times\mathbb{R}^3\\
u(0)=u_0\in H_x^s(\mathbb{R}^3),
\end{cases}
\end{aligned}
\end{equation}
where $s\geq 1/2$.

It is known that there is mass conservation law for (\ref{equ}),
i.e.,
\begin{equation}
M(u(t))=\int|u(t,x)|^2dx=M(u(0)).
\end{equation}
If $s\geq 1$, there is also energy conservation law,
\begin{equation}
E(u(t))=\dfrac{1}{2}\int |\nabla
u(t,x)|^2dx+\dfrac{1}{4}\int|u(t,x)|^4dx=E(u(0)).
\end{equation}
Moreover, (\ref{equ}) is locally well-posed for $s>1/2$. In
particular, there is blow up criteria for (\ref{equ}): If $s>1/2$ and $u$ is the
solution to (\ref{equ}) with maximal existence interval $[0,T^*)$,
then if $T^*<\infty$,
\begin{equation}\label{blow-up}
\lim\limits_{t\uparrow T^*}||u(t)||_{H^s}=\infty.
\end{equation}

Thus global well-posedness of (\ref{equ}) for $s\geq 1$(see
\cite{CW}) follows immediately from energy conservation law.
Scattering in energy space or above is proved by Ginibre and Velo in
\cite{GV}. However, for $s<1$, there is no energy conservation. More
precisely, there is no known coercive quantity that can be used to
control the $H^s$ norm, which is the main obstruction for global
well-posedness and scattering. It was conjectured by the following

\noindent \textbf{Conjecture. } Let $s\geq 1/2$, then (\ref{equ})
is globally well-posed in $H^s(\mathbb{R}^3)$ and there is
scattering.

\begin{remark}
The two dimensional defocusing, cubic NLS analogy of this conjecture
has been solved by Dodson\cite{ben2} recently. He showed that the
defocusing, cubic NLS is globally well-posed and there is scattering
in $L^{2}(\mathbb{R}^2)$.
\end{remark}
 The conjecture has attracted much attentions.
 Previous
work can be found in
\cite{B1},\cite{CKST1},\cite{CKST2},\cite{ben1},\cite{KM}. We state
these results briefly.

The breakthrough work was made by Bourgain(see
\cite{B1},\cite{B2},\cite{B3}). He used the Fourier truncation
method to capture the smoothing effect of the nonlinearity. He
proved global well-posedness for $s>11/13$ and scattering for
radially symmetry data $u_0\in H^s(\mathbb{R}^3)$ with $s>5/7$.

Inspired by the Fourier truncation method, Colliander, Keel,
Staffilani, Takaoka, and Tao introduced the I-method( or almost
conservation law method) in \cite{CKST1}, which is a smoothed
version of the Fourier truncation method.  By smoothing out the
rough data, they can make use of the energy conservation law.
Indeed, they proved almost conservation law for the smoothed
solution via multilinear estimate, and then proved a polynomial
bound for the solution of (\ref{equ}) for $s>5/6$, thus obtained
global well-posedness for $s>5/6$, but not the scattering result.

To weaken the regularity requirement in \cite{CKST1} for global
well-posedness and radical symmetry assumption in \cite{B1} for
scattering, Colliander, Keel, Staffilani, Takaoka, and
Tao\cite{CKST2} proved a new type Morawetz inequality. Together with
the I-method, they are able to bound the solution in
$H^{s}(\mathbb{R}^3)$ and $L_{t,x}^4$ uniformly provided $s>4/5$,
thus they are able to prove global well-posedness and scattering for
$s>4/5$.

Recently, Dodson\cite{ben1} improved the result in \cite{CKST2} via
linear-nonlinear decomposition method introduced by Roy\cite{roy}.
By using linear-nonlinear decomposition, I-method, and together with
double layer decomposition, he was able to show globall
well-posedness and scattering for $s>5/7$.

 On the other hand, Kenig and Merle in \cite{KM2} used the
 concentration-compactness method to deal with global
well-posedness and
 scattering problems at critical regularity. By
 profile decomposition and concentration compactness/rigidity argument, they
 showed in \cite{KM} that in order to prove Conjecture, it suffices to
 bound the solution in $\dot{H}^{1/2}$.

In this paper, we adopt an idea of large time iteration. Normally,
in order to obtain global well-posedness, we would obtain local
well-posedness on a small time interval, and then use iteration
method to extend the local solution to global one. Roughly speaking,
for each iteration, we extend the solution on time interval by one
unit. Such iteration is 'slow' in some sense. Thus we would like to have a
'faster' iteration strategy, where the iterates on time interval are larger
than one for each iteration. As a consequence, the number of iterations is
heavily reduced.

To see how such an idea works, we combine the idea of
linear-nonlinear decomposition used by Dodson in \cite{ben1} and Roy
in \cite{roy}, the idea of modified energy via resonance
decomposition in \cite{CKSTT4}, and the idea of 'large time
iteration'. It is captured that the nonlinear part
of the solution enjoys more regularity in high frequency. Thus we
can make use of such a smoothing effect by linear-nonlinear
decomposition. Furthermore,  by adding a correction term to the
energy functional $E(Iu)$, we can obtain a better control of the
increment of the energy(see \cite{CKSTT4} for more discussion). Thus
we are able to prove a refined version of almost conservation law.
Finally, by large time iteration, we are able to reduce the amount of
iterations. The main result of this paper is the following
\begin{theorem}\label{main}
(\ref{equ}) is globally well-posed and there is scattering in
$H^s(\mathbb{R}^3)$ for $s>49/74$.
\end{theorem}
This paper is organized as follows: In Section 2, we set some
notations and recall some preliminary facts. In section 3 and 4, we prove
a local existence theorem and an smoothing
effect of the nonlinear part of the solution, respectively. In section 5, we
recall the construction of modified energy in \cite{CKSTT4} and prove a refined almost conservation law.
Theorem \ref{main} will be proved in the last section.

\section{Notations and Preliminaries}
Given $A,B\geq 0$, by $A\lesssim B$ we mean $A\leq C\cdot B$ for
some universal constant $C$. By $A\sim B$ it means $A\lesssim B$ and
$B\lesssim A$. The notation $A\gtrsim B$ means $B\lesssim A$. The
notation $A\ll B$ means $A\leq K\cdot B$ for some large universal
constant $K$. The notation $A\gg B$ means $A\geq K\cdot B$ for some
large constant $K>0$. The notation $A+$ means $A+\epsilon$ for some
universal $0<\epsilon\ll 1$. And the notation $A-$ means
$A-\epsilon$ for some universal $0<\epsilon\ll 1$. By $<a>$ we mean
$(1+|a|^2)^{1/2}$.
\begin{definition}
Let $1\leq q,r\leq\infty$, we say that (q,r) is admissible if
$$\dfrac{2}{q}=3(\dfrac{1}{2}-\dfrac{1}{r}).$$
\end{definition}

We recall the definition of I-operator, which is a Fourier
multiplier.
\begin{definition}
The I-operator $I_N:H^s(\mathbb{R}^3)\rightarrow H^1(\mathbb{R}^3)$
is defined as
$$\widehat{I_Nu}(\xi)=m_N(\xi)\hat{u}(\xi),$$
where $m$ is smooth, radially symmetric, and satisfies
\begin{align*}
m_N(\xi)=\begin{cases} 1,~~~~~&|\xi|\leq N\\
(\frac{N}{|\xi|})^{1-s}, &|\xi|>2N.
\end{cases}
\end{align*}
\end{definition}
We abbreviate $I_N,m_N$ as $I,m$, respectively. For the convenience
of the readers, we list some basic facts of the $I$-operator and
explain how the $I$-method works. For more details, the reader can
refer to, for example,
\cite{CKST1}\cite{CKST2}\cite{CKSTT4}\cite{ben1}\cite{roy}. We have
the estimates
\begin{equation*}
||\nabla Iu||_{L^2(\mathbb{R}^3)}\lesssim
N^{1-s}||u||_{\dot{H}^s(\mathbb{R}^3)}.
\end{equation*}
\begin{equation*}
||u||_{\dot{H}^s(\mathbb{R}^3)}\lesssim
||Iu||_{\dot{H}^1(\mathbb{R}^3)}.
\end{equation*}
Therefore, $||u(t)||_{\dot{H}^s(\mathbb{R}^3)}$ is controlled by
$E(Iu(t))$:
$$||u(t)||_{\dot{H}^s(\mathbb{R}^3)}\lesssim E(Iu(t)).$$
Thus, we are reduced to controlling the modified energy $E(Iu)$.
Note that we can write $E(Iu)$ in a multilinear form:
$$E(Iu)=\Lambda_2(\sigma_2;u)+\Lambda_4(\sigma_4;u),$$
 where
$\Lambda_2, \Lambda_4$ are some multilinear functionals and
$\sigma_2, \sigma_4$ are some symbols, see section 5.1 for the
definition, and see \cite{CKSTT4} for more details.

To obtain a better control on $E(Iu)$, we add a correction term to
$E(Iu)$ to construct another modified energy functional
$\tilde{E}(u(t))$ such that $\tilde{E}(u(t))$ has slower energy
increment. Similar to \cite{CKSTT4}, we use resonance decomposition
to construct $\tilde{E}$ as
$$\tilde{E}(u(t)):=\Lambda_2(\sigma_2;u)+\Lambda_4(\tilde{\sigma}_4;u),$$
where $\tilde{\sigma_4}$ is defined via resonance decomposition. See
section 5.1, also see \cite{CKSTT4} for more details. By such
construction, we are reduced to controlling $\tilde{E}(u(t))$.

 Let $u$ be a solution to (\ref{equ}) on time interval
$J=[t_0,T]$ such that $u(t_0)=u_0$. We know that
$\forall~t\in[t_0,T]$, the Duhamel identity holds:
\begin{equation}\label{du}
u(t)=e^{it\Delta}u_0-i\int_{t_0}^t e^{i(t-s)\Delta}(|u|^2u)(s)ds.
\end{equation}
We then decompose $u$ into linear part $u_J^l$ and nonlinear part
$u_J^{nl}$ adapted to $J$, i.e.,
\begin{equation}\label{l-n}
u_J^{l}(t):=e^{it\Delta}u(t_0),~u_J^{nl}(t):=-i\int_{t_0}^t
e^{i(t-s)\Delta}(|u|^2u)(s)ds.
\end{equation}
In later sections, if there is no cause of confusion, we simply
write $u_J^l, u_J^{nl}$ as $u^l, u^{nl}$, respectively.

 We need some Littlewood-paley theory, see \cite{stein}, \cite{tao1} for example.
Let $\phi(\xi)$ be a fixed radial bump function adapted to the ball
$\{\xi:|\xi|\leq 2\}$ which equals 1 on the ball $\{\xi:|\xi|\leq
1\}$. Let $N$ be a dyadic number. Define the Fourier multipliers
\begin{align*}
&\widehat{P_{<N}u}(\xi):=\phi(\dfrac{\xi}{N})\hat{u}(\xi),\\
&\widehat{P_{>N}u}(\xi):=(1-\phi(\dfrac{\xi}{N}))\hat{u}(\xi),\\
&\widehat{P_{N}u}(\xi):=(\phi(\xi/N)-\phi(2\xi/N))\hat{u}(\xi).
\end{align*}
Similarly, we can define $P_{\geq N},~P_{\leq N}$.

 In the
following, we state some facts that will be used frequently in later
sections.

The first one is the Bernstein type inqualities.
\begin{prop}\label{bernstein}\cite{tao1}
Let $s\geq 0$ and $d$ a positive integer. $1\leq p\leq q\leq
\infty$. Then
\begin{align*}
&||P_{\geq
N}u||_{L_x^p(\mathbb{R}^d)}\lesssim_{p,s,d}N^{-s}||\nabla^sP_{\geq
N}u||_{L_x^p(\mathbb{R}^d)};
\\
&||P_{\leq N}\nabla^s
u||_{L_x^p(\mathbb{R}^d)}\lesssim_{p,s,d}N^{s}||P_{\leq
N}u||_{L_x^p(\mathbb{R}^d)};\\
&||P_{ N}\nabla^{\pm
s}u||_{L_x^p(\mathbb{R}^d)}\lesssim_{p,s,d}N^{\pm
s}||P_{N}u||_{L_x^p(\mathbb{R}^d)};\\
&||P_{\leq
N}u||_{L_x^q(\mathbb{R}^d)}\lesssim_{p,s,d}N^{\frac{d}{p}-\frac{d}{q}}||P_{\leq
N}u||_{L_x^p(\mathbb{R}^d)};
\\
&||P_{
N}u||_{L_x^q(\mathbb{R}^d)}\lesssim_{p,s,d}N^{\frac{d}{p}-\frac{d}{q}}||P_{
N}u||_{L_x^p(\mathbb{R}^d)}.
\end{align*}
\end{prop}

Next we state Strichartz estimate, which is fundamental to the study
of dispersive equation. The reader can refer to \cite{CA1} and
\cite{tao1}
 for more details.
\begin{lemma}
Let $(q,r)$ be admissible. Let $u$ be a solution to (\ref{equ}) on
time interval $J=[t_0,T]$ with initial data $u(t_0)=u_0$, which
satisfies the Duhamel identity,
$$u(t)=e^{it\Delta}u_0-i\int_{t_0}^t e^{i(t-s)\Delta}|u|^2u(s)ds.$$
 Then we have
\begin{equation}\label{str}
||e^{it\Delta}u||_{L_t^q(J)L_x^r}\lesssim ||u_0||_{L_x^2},~~||\int_J
e^{i(t-s)\Delta}|u|^2u(s)ds||_{L_t^q(J)L_x^r}\lesssim
|||u|^2u||_{L_t^{\tilde{q}'}(J)L_x^{\tilde{r}'}},
\end{equation}
where $(\tilde{q},\tilde{r})$ is admissible and
$$\dfrac{1}{\tilde{q}}+\dfrac{1}{\tilde{q}'}=1,~\dfrac{1}{\tilde{r}}+\dfrac{1}{\tilde{r}'}=1.$$
\end{lemma}

\begin{definition}
Let $J$ be a time interval. Define
$$Z_I(J;u):=\sup\limits_{(q,r)\text{~admissible}}||\nabla
Iu||_{L_t^q(J)L_x^r(\mathbb{R}^3)}.
$$
\end{definition}

\section{Local Existence}
We need a simple lemma.
\begin{lemma}\label{common}
Let $\delta<s$ and $(q,r)$ be admissible pair. Then
$$||\nabla^{\delta} P_{\geq N}u||_{L_t^qL_x^r}\lesssim
N^{\delta-1}||\nabla Iu||_{L_t^qL_x^r}.
$$
\end{lemma}

The proof is standard by Littlewood-Paley decomposition. We omit the details and leave the proof to the reader.

We also need a local existence result, whose proof can be found in
\cite{CKST2}.
\begin{lemma}\label{local-existence}
Consider $u(t,x)$ be as in (\ref{equ}) defined on $J\times
\mathbb{R}^3$. Assume
\begin{equation}
||u||_{L_{t,x}^4(J\times\mathbb{R}^3)}\leq \epsilon,
\end{equation}
for some small constant $\epsilon>0$. Assume $u_0\in
C_0^{\infty}(\mathbb{R}^3))$. Then for $s>1/2$ and sufficiently
large $N$, we have
\begin{equation}
Z_I(J,u)\leq C(||u_0||_{\dot{H}^s}).
\end{equation}
\end{lemma}

The following local existence is a modification of Lemma
\ref{local-existence}. In Lemma \ref{local-existence}, the
$L_{t,x}^4$ norm is assumed to be small, while, for our purpose, we remove the
smallness assumption. In some sense, such a local
existence can be viewed as a large time existence and the iteration
based on such a local existence can be viewed as a large time iteration.
\begin{lemma}\label{lemma1}
(Modified local existence) Let $u$ be a solution to (\ref{equ}) on
time interval $J=[0,\tau]$. Assume
$$\sup\limits_{t\in J}E(Iu(t))\lesssim
1,~~~||u||_{L_{t,x}^4(J\times\mathbb{R}^3)}< \infty.$$ Then for
admissible pair $(q,r)$,
\begin{align*}
&Z_I(J;u^{l})\lesssim 1;\\
&||\nabla Iu^{nl}||_{L_t^q(J)L_x^r}\lesssim
\max\{1,||u||_{L_{t,x}^4}^4\}^{1/q};\\
&||\nabla Iu||_{L_t^q(J)L_x^r}\lesssim
\max\{1,||u||_{L_{t,x}^4}^4\}^{1/q}.
\end{align*}
\end{lemma}
\begin{proof}
It is clear that by Strichartz estimate, we have
$$Z_I(J;u^{l})\lesssim ||\nabla Iu_0||_{L_x^2}\lesssim 1.$$
Thus by triangle inequality, it suffices to show that
$$||\nabla Iu||_{L_t^q(J)L_x^r}\lesssim \max\{1,||u||_{L_{t,x}^4}^4\}^{1/q}.$$

\noindent We decompose $J$ into subintervals $J_1,...,J_m$ such that
for each subinterval we have
$$||u||_{L_{t,x}^4(J_k\times\mathbb{R}^3)}^4\leq \epsilon$$
for some small constant $\epsilon>0$. Thus, $m$ is essentially
$||u||_{L_{t,x}^4}^4$. Since for each $J_k$
$$||\nabla Iu||_{L_t^q(J_k)L_x^r}^q\lesssim 1,$$
summing over $k$ yields
$$||\nabla Iu||_{L_t^q(J)L_x^r}^q\lesssim ||u||_{L_{t,x}^4(J\times\mathbb{R})}^4.$$
\end{proof}

\begin{definition}
We define
$$M(J,u,q):=\max\{1,||u||_{L_{t,x}^4(J\times\mathbb{R}^3)}^4\}^{1/q}.$$
\end{definition}

\section{Smoothing effect of nonlinearity}
In this section, we prove a smoothing effect of the nonlinearity,
which is crucial to prove the almost conservation law in next section.

The following Lemma was proved by Dodson\cite{ben1}.

\begin{lemma}\label{smooth}
Let $u$ be a solution to (\ref{equ}) on time interval $J=[0,T]$ such
that
$$||u||_{L_{t,x}^4(J\times\mathbb{R}^3)}\leq\epsilon,~~||\nabla Iu_0||_{L_x^2}\leq 1.$$
Let $N_j$ be a dyadic number. Then if $N_j\lesssim N$,
\begin{equation}
||P_{>N_j}\nabla Iu^{nl}||_{L_t^qL_x^r}\lesssim
N_j^{-1/2},~~||P_{>N_j}\nabla Iu^{nl}||_{L_t^{\infty}L_x^2}\lesssim
N_j^{-1}.
\end{equation}
and If $N_j\gtrsim N$,
\begin{equation}
||P_{>N_j}\nabla Iu^{nl}||_{L_t^qL_x^r}\lesssim
N^{-1/2},~~||P_{>N_j}\nabla Iu^{nl}||_{L_t^{\infty}L_x^2}\lesssim
N^{-1}.
\end{equation}
\end{lemma}

By Lemma \ref{smooth} and interpolation, we obtain the following
smoothing effect.
\begin{theorem}\label{inter}
Suppose $J$ is an interval such that
\begin{equation}\sup\limits_{t\in J}E(Iu(t))\lesssim
1,~~||u||_{L_{t,x}^4(J\times\mathbb{R}^3)}<\infty.
\end{equation}
For any admissible pair $(q,r)$ with $q\geq 4$, then if $N_j\lesssim
N$,
\begin{equation}
||P_{>N_j}\nabla Iu^{nl}||_{L_t^qL_x^r}\lesssim
N_j^{-\frac{3}{4}-\frac{\theta}{4}}M(J,u,q)
\end{equation}
and if $N_j\gtrsim N$,
\begin{equation}
||P_{>N_j}\nabla Iu^{nl}||_{L_t^qL_x^r}\lesssim
N^{-\frac{3}{4}-\frac{\theta}{4}}M(J,u,q),
\end{equation}
where $\theta$ satisfies
\begin{align*}
\begin{cases}
\frac{1}{q}=\frac{\theta}{\infty}+\frac{1-\theta}{4}=\frac{1}{4}-\frac{\theta}{4}\\
\frac{1}{r}=\frac{\theta}{2}+\frac{1-\theta}{3}~=\frac{1}{3}+\frac{\theta}{6}.
\end{cases}
\end{align*}
\end{theorem}
\begin{proof}
We only prove the case that $N_j\lesssim N$. First, by the
interpolation between $L_t^{\infty}L_x^2$ and $L_t^2L_x^6$ with
\begin{align*}
\begin{cases}
\frac{1}{4}=\frac{\theta}{\infty}+\frac{1-\theta}{2}=\frac{1}{2}-\frac{\theta}{2}\\
\frac{1}{3}=\frac{\theta}{2}+\frac{1-\theta}{6}~=\frac{1}{6}+\frac{\theta}{3},
\end{cases}
\end{align*}
we get $\theta=1/2$. Thus by the interpolation we have
\begin{align*}
||P_{>N_j}\nabla Iu^{nl}||_{L_t^4L_x^3}\lesssim & ||P_{>N_j}\nabla
Iu^{nl}||_{L_t^{\infty}L_x^2}^{1/2}||P_{>N_j}\nabla
Iu^{nl}||_{L_t^2L_x^6}^{1/2}\\
\lesssim &
N_j^{-1/2}N_j^{-\frac{1}{2}\times\frac{1}{2}}M(J,u,2)^{1/2}\\
\lesssim & N_j^{-3/4}M(J,u,4).
\end{align*}
Secondly, observe that for each admissible pair $(q,r)$ with
$q\geq 4$, we have
\begin{align*}
\begin{cases}
\frac{1}{q}=\frac{\theta}{\infty}+\frac{1-\theta}{4}=\frac{1}{4}-\frac{\theta}{4}\\
\frac{1}{r}=\frac{\theta}{2}+\frac{1-\theta}{3}~=\frac{1}{3}+\frac{\theta}{6}
\end{cases}
\end{align*}
for some $0\leq \theta\leq 1$. Thus
\begin{align*}
||P_{>N_j}\nabla Iu^{nl}||_{L_t^q L_x^r}\lesssim & ||P_{>N_j}\nabla
Iu^{nl}||_{L_t^{\infty}L_x^2}^{\theta}||P_{>N_j}\nabla
Iu^{nl}||_{L_t^4L_x^3}^{1-\theta}\\
\lesssim &
N_j^{-\theta}N_j^{-\frac{3}{4}(1-\theta)}M(J,u,4)^{1-\theta}\\
\lesssim & N_j^{-3/4-\theta/4}M(J,u,4/(1-\theta))\\
\lesssim & N_j^{-3/4-\theta/4}M(J,u,q).
\end{align*}
\end{proof}

\section{Modified energy functional and almost conservation law}
In this section, we recall the construction of modified energy
functional $\tilde{E}$ in \cite{CKSTT4}. We prove a refined version
of almost conservation law. We show
\begin{theorem}\label{modified}
(Existence of an almost conserved quantity)  Assume $u$ is a smooth
in time, schwartz in space solution to (\ref{equ}) with initial data
$u_0\in H_x^{s}(\mathbb{R}^3)(s>1/2)$ defined on
$J\times\mathbb{R}^3$ such that
\begin{equation}
||u||_{L_{t,x}^4(J\times\mathbb{R}^3)}<\infty,~\sup\limits_{t\in
J}E(Iu(t))\lesssim 1,
\end{equation}
then there exists a functional
$\tilde{E}=\tilde{E}_N:\mathcal{S}_x(\mathbb{R}^3)\rightarrow
\mathbb{R}$ defined on Schwartz functions $u\in
\mathcal{S}_x(\mathbb{R}^3)$ with the following properties.

(1) (Fixed-time bounds) For any $u\in\mathcal{S}_x(\mathbb{R}^3)$ ,
\begin{equation}\label{fixed-time}
|E(Iu)-\tilde{E}(u)|\lesssim N^{-1/8+}.
\end{equation}

(2) (Almost conserved law)
\begin{equation}\label{conserved}
\sup\limits_{t\in J}|\tilde{E}(u(t))-\tilde{E}(u_0)|\lesssim
N^{-9/8+}\max\{1, \dfrac{M(J,u,2)}{N^{1-}},
\dfrac{M(J,u,1)}{N^{2-}}\}.
\end{equation}

\end{theorem}
In section \ref{recall}, we recall the construction of modified
energy functional $\tilde{E}$ via resonance decomposition. The proofs of pointwise estimate (\ref{fixed-time}) and
the almost conservation law(\ref{conserved}) are given in section 5.2 and 5.3, respectively.
\subsection{Construction of modified
energy via resonance decomposition\cite{CKSTT4}}\label{recall} In
this section, we recall the construction of modified energy via
resonance decomposition in \cite{CKSTT4}. The construction of
modified energy functional $\tilde{E}$ in \cite{CKSTT4} is on
$\mathbb{R}^2$, which can be extended to $\mathbb{R}^3$ without
any change.

Let $k$ be an integer. Denote the space
$$\Sigma_k:=\{(\xi_1,...,\xi_k)\in (\mathbb{R}^3)^k\mid
\xi_1+...+\xi_k=0\}.$$

Let $M:\Sigma_k\rightarrow \mathbb{C}$ be a smooth tempered symbol,
and $u_1,...,u_k\in \mathcal{S}(\mathbb{R}^3)$, define the
$k$-functional
$$\Lambda_k(M;u_1,...,u_k):=Re\int_{\Sigma_k}M(\xi_1,...,\xi_k)\widehat{u_1}(\xi_1)...\widehat{u_k}(\xi_k).$$
If $k$ is even, we abbreviate
$\Lambda_k(M;u):=\Lambda_k(M;u,\bar{u},...,u,\bar{u})$. Let $k$ be
an even number and set $A:=\{1,3,...,k-1\}$, $B:=\{2,4,...,k\}$. Let
$h$ be the operator be defined by
$$h(M(\xi_1,\xi_2,...,\xi_{k-1},\xi_k)):=\overline{M}(\xi_2,\xi_1,...,\xi_k,\xi_{k-1}).$$
Let $S(A)$ and $S(B)$ be symmetric groups on $A$ and $B$,
respectively. Let $H:=\{h,id\}$ be a group of two elements, where
$id$ is the identity map on $\Sigma_k$(hence on the space of
tempered symbols). Define $G_k$ to be the group generated by $S(A),
S(B)$ and $H$. Then $|G_k|=2(k/2)!(k/2)!$. Define
$[M]_{\text{sym}}:=\dfrac{1}{|G_k|}\sum\limits_{g\in G_k} gM$. Then
$$\Lambda_k(M;u)=\Lambda_k([M]_{sym};u).$$

Define the extended symbol $X(M)$ by
$$X(M)(\xi_1,...,\xi_k):=M(\xi_{123},\xi_4,...,\xi_{k+2}),$$
where $\xi_{123}:=\xi_1+\xi_2+\xi_3$. Similarly, denote $\xi_{ab}=\xi_a+\xi_b$. Set
$$\alpha_4:=2\xi_{12}\cdot\xi_{14}=-2|\xi_{12}||\xi_{14}|\text{cos}\angle
(\xi_{12},\xi_{14}),
~\sigma_2(\xi_1,\xi_2):=\dfrac{1}{2}|\xi_1|^2m_1^2.$$

Let $\theta_0$ be a small parameter to be determined later. Define
the non-resonant set
\begin{align*}
\Omega_{nr}:=\Omega_1\cup \Omega_2,
\end{align*}
where
$$\Omega_1=\{(\xi_1,\xi_2,\xi_3,\xi_4)\in\Sigma_4\mid
\max\limits_{1\leq j\leq 4}|\xi_j|\leq N\},
$$
and $$\Omega_2= \{(\xi_1,\xi_2,\xi_3,\xi_4)\in \Sigma_4\mid
|\text{cos}\angle(\xi_{12},\xi_{14})|\geq \theta_0\}.
$$

The symbol $[X(\sigma_2)]_{\text{sym}}$ is given by
$$
[2iX(\sigma_2)]_{\text{sym}}=\dfrac{i}{4}\sum\limits_{j=1}^4(-1)^{j-1}m_j^2
|\xi_j|^2.
$$
Define the modified energy functional
\begin{equation}
\tilde{E}(u):=\Lambda_2(\sigma_2;u)+\Lambda_4(\tilde{\sigma}_4;u),
\end{equation}
where
\begin{equation}\label{tilde}
\tilde{\sigma}_4:=\dfrac{[2iX(\sigma_2)]_{\text{sym}}}{i\alpha_4}1_{\Omega_{nr}}.
\end{equation}
\begin{remark}
Note that $$E(Iu)=\Lambda_2(\sigma_2;u)+\Lambda_4(\sigma_4;u).$$ Thus
\begin{equation}\label{p-p}
E(Iu)-\tilde{E}(u)=\Lambda_4(\sigma_4-\tilde{\sigma}_4;u).
\end{equation}
Also note that
\begin{equation}
\begin{aligned}
&\tilde{E}(u(t))-\tilde{E}(u(0))\\
=&\int_0^{t}
\Lambda_4([-2iX(\sigma_2)]_{\text{sym}}+i\tilde{\sigma}_4\alpha;u(t'))dt'+\int_0^{t}
\Lambda_6([4iX(\tilde{\sigma}_4)]_{\text{sym}};u(t'))dt'.
\end{aligned}
\end{equation}
\end{remark}
\subsection{Pointwise Estimate}\label{point}

In this section, we obtain a pointwise estimate on the modified
energy functional $\tilde{E}$. We prove the following proposition,
whose analogy in $\mathbb{R}^2$ can be found in \cite{CKSTT4}.
\begin{prop}\label{point}
Let $u\in \mathcal{S}(\mathbb{R}^3)$ be a Schwartz function, then we
have
\begin{equation}
|E(Iu)-\tilde{E}(u)|\lesssim N^{-1+}\theta_0^{-1}||\nabla
Iu||_{L_x^2(\mathbb{R}^3)}^4.
\end{equation}
\end{prop}
To prove Proposition \ref{point}, we need the following lemma, whose
proof can be found in \cite{CKSTT4}.
\begin{lemma}\label{pp}
For any $(\xi_1,\xi_2,\xi_3,\xi_4)\in\Sigma_4$, we have
$$|\sigma_4-\tilde{\sigma}_4|\lesssim \dfrac{min(m_1,m_2,m_3,m_4)^2}{\theta_0}.$$
\end{lemma}
\begin{proof}[Proof of Proposition  \ref{point}] By (\ref{p-p}), it suffices to
show the following estimate
$$
\int_{\Sigma_4}|\sigma_4-\tilde{\sigma}_4||\hat{u}(\xi_1)\hat{u}(\xi_2)\hat{u}(\xi_3)\hat{u}(\xi_4)|\lesssim
N^{-1+}\theta_0^{-1}||\nabla Iu||_{L_x^2}.
$$

To do this, we decompose $u$ into dyadic pieces $u_j$, where $u_j$ is localized
with a smooth cutoff function in spatial frequency space having
support $|\xi|\sim 2^{k_j}\equiv N_j, k_j\in\mathbb{Z}$. By
symmetry, we can assume $N_1\geq N_2\geq N_3\geq N_4$. Furthermore,
we can assume $N_1\sim N_2\geq N$.

So it suffices to show that
\begin{equation}\label{dyadic}
I_1:=m(N_1)^2\int_{\Sigma_4}\prod\limits_{j=1}^4 u_j\leq
C(N_1,N_2,N_3,N_4)N^{-1+}||\nabla Iu_j||_{L_x^2},
\end{equation}
where $C(N_1,N_2,N_3,N_4)$ is sufficient small constant such that we
can sum over $N_1$,$N_2$,\\$N_3$,$N_4$. Without loss of generality,
we assume $u_i(i=1,2,3,4)$ is real and nonnegative. To this end, we
consider the following cases.

Case 1. $N_4\gtrsim 1$.
\begin{align*}
I_1\lesssim &
m(N_1)^2||u_1||_{L_x^3}||u_2||_{L_x^3}||u_3||_{L_x^6}||u_1||_{L_x^6}\\
\lesssim &
m(N_1)^2||\nabla^{1/2}u_1||_{L_x^2}||\nabla^{1/2}u_2||_{L_x^2}||\nabla
u_3||_{L_x^2}||\nabla u_4||_{L_x^2}\\
\lesssim & N_1^{-1/2}N_2^{-1/2}m(N_3)^{-1}m(N_4)^{-1}||\nabla
Iu_1||_{L_x^2}||\nabla Iu_2||_{L_x^2}||\nabla Iu_3||_{L_x^2}||\nabla
I u_4||_{L_x^2}\\
\lesssim & N_1^{-}N^{-1+}||\nabla Iu||_{L_x^2}^4.
\end{align*}

Case 2. $N_1\geq N_2\geq N_3\gtrsim 1\gg N_4.$

For each fixed $\xi_4$ such that $|\xi_4|\sim N_4$, let
$$\Omega_{\xi_4}=\{(\xi_1,\xi_2,\xi_3)\in\mathbb{R}^3\times \mathbb{R}^3\times \mathbb{R}^3\mid
\xi_1+\xi_2+\xi_3+\xi_4=0\}.$$ Then we have
\begin{align*}
I_1=&m(N_1)^2\int_{|\xi_4|\sim
N_4}\Big\{\int_{\Omega_{\xi_4}}\hat{u}_1\hat{u}_2\hat{u}_3
d\xi_1d\xi_2d\xi_3\Big\} \hat{u}_4d\xi_4\\
\lesssim & m(N_1)^2\Big(\int_{|\xi_4|\sim N_4}
\hat{u}_4d\xi_4\Big)\sup\limits_{\xi_4:|\xi_4|\sim
N_4}\Big\{\int_{\Omega_{\xi_4}}\hat{u}_1\hat{u}_2\hat{u}_3
d\xi_1d\xi_2d\xi_3\Big\}\\
\lesssim &
m(N_1)^2||u_4||_{L_x^2}\Big[\mu(\{\xi_4\in\mathbb{R}^3\mid
|\xi_4|\sim N_4\})\Big]^{1/2}\sup\limits_{|\xi_4|\sim
N_4}\Big\{\int_{\Omega_{\xi_4}}\hat{u}_1\hat{u}_2\hat{u}_3
d\xi_1d\xi_2d\xi_3\Big\}\\
\lesssim & m(N_1)^2N_4^{1/2} ||\nabla
Iu_4||_{L_x^2}||u_1||_{L_x^{12/5}}||u_2||_{L_x^{12/5}}||u_3||_{L_x^6}\\
\lesssim & m(N_1)^2N_4^{1/2} ||\nabla
Iu_4||_{L_x^2}||\nabla^{1/4}u_1||_{L_x^{2}}||\nabla^{1/4}u_2||_{L_x^{2}}||\nabla
u_3||_{L_x^2}\\
\lesssim & N_1^{0-}N_4^{1/2}N^{-3/2+}||\nabla Iu||_{L_x^2}^4.
\end{align*}

Case 3. $N_3\ll 1$.

Similar to the argument in Case 2, let
$$\Omega_{\xi_3,\xi_4}:=\{(\xi_1,\xi_2)\in\mathbb{R}^3\times\mathbb{R}^3\mid
\xi_1+\xi_2+\xi_3+\xi_4=0\}.$$

Then we obtain
\begin{align*}
I_1=&m(N_1)^2\int_{|\xi_4|\sim N_4}\int_{|\xi_3|\sim
N_4}\Big\{\int_{\Omega_{\xi_3,\xi_3}}\hat{u}_1\hat{u}_2
d\xi_1d\xi_2\Big\} \hat{u}_3\hat{u}_4d\xi_3d\xi_4\\
\lesssim & m(N_1)^2N_3^{1/2}||\nabla Iu_3||_{L_x^2}N_4^{1/2}||\nabla
Iu_4||_{L_x^2}||u_1||_{L_x^2}||u_2||_{L_x^2}\\
 \lesssim &
N_1^{0-}N_4^{1/2}N^{-2+}||\nabla Iu||_{L_x^2}^4.
\end{align*}
The proof of Proposition \ref{point} is concluded.
\end{proof}

\subsection{Almost Conservation Law}\label{almost}
In this section we prove an almost conservation law for the modified
energy functional $\tilde{E}$, which is crucial to establish global
well-posedness and scattering.
\begin{prop}\label{conservation}
(Almost conservation law). Let $J=[0,T]$. Let $u$ be a smooth in
time, schwartz in space solution to (\ref{equ}) with initial data
$u_0\in H_x^{s}(\mathbb{R}^3)(s>1/2)$ defined on
$J\times\mathbb{R}^3$ such that
\begin{equation}
\sup\limits_{t\in J}E(Iu(t))\leq
1,~||u||_{L_{t,x}^4(J\times\mathbb{R}^3)}< \infty,
\end{equation}
then we have the quadrilinear estimate
\begin{equation}\label{quad}
|\int_0^{t_0}
\Lambda_4([-2iX(\sigma_2)]_{\text{sym}}+i\tilde{\sigma}_4\alpha;u(t))dt|\lesssim
N^{-9/8+}\max\{1,\dfrac{M(J,u,2)}{N^{1-}},\dfrac{M(J,u,1)}{N^{2-}}\}
\end{equation}
and the sextilinear estimate
\begin{equation}\label{sextile}
|\int_0^{t_0}
\Lambda_6([4iX(\tilde{\sigma}_4)]_{\text{sym}};u(t))dt|\lesssim
N^{-9/8+}\max\{1,\dfrac{M(J,u,2)}{N^{1-}},\dfrac{M(J,u,1)}{N^{2-}}\}.
\end{equation}
\end{prop}

\subsubsection{Sextilinear Estimate} Now we prove the sextilinear
estimate. First we show the following lemma.
\begin{lemma}\label{sexti}
Let $J=[0,T]$. Let $u$ be a smooth in time, schwartz in space
solution to (\ref{equ}) with initial data $u_0\in
H_x^{s}(\mathbb{R}^3)(s>1/2)$ defined on $J\times\mathbb{R}^3$ such
that
\begin{equation}\label{assumption}
\sup\limits_{t\in J}E(Iu(t))\leq
1,~||u||_{L_{t,x}^4(J\times\mathbb{R}^3)}< \infty,
\end{equation}
then
\begin{equation}\label{sextile-sch}
|\int_0^{T}
\Lambda_6([4iX(\tilde{\sigma}_4)]_{\text{sym}};u(t))dt|\lesssim
\theta_0^{-1}N^{-2+}\max\{1,\dfrac{M(J,u,2)}{N^{1-}},
\dfrac{M(J,u,1)}{N^{2-}}\}.
\end{equation}
\end{lemma}
\begin{proof}
We may assume that $\max\limits_{1\leq j\leq 6}\{|\xi_j|\}\geq N/3$,
otherwise the symbol $[4iX(\tilde{\sigma})]_{\text{sym}}$
vanishes(recall that if $\max\limits_{1\leq j\leq 6}\{|\xi_j|\}<
N/3$, then $4X(\tilde{\sigma}_4)=1$). With such assumption, we then
remove the symmetry of the symbol. It suffices to show that
\begin{equation}\label{sextile-sch}
|\int_0^{T} \Lambda_6(4iX(\tilde{\sigma}_4);u(t))dt|\lesssim
\theta_0^{-1}N^{-2+}\max\{1,\dfrac{M(J,u,2)}{N^{1-}},
\dfrac{M(J,u,1)}{N^{2-}}\}.
\end{equation}

By lemma \ref{pp}, we have
$$|X(\tilde{\sigma}_4)|\lesssim \dfrac{1}{\theta_0}min\{m_{123},m_4,m_5,m_6\}^2.$$
If we arrange $\xi_1,...,\xi_6$ as $\xi_1^*,...,\xi_6^*$ such that
$|\xi_1^*|\geq |\xi_2^*|\geq...\geq |\xi_6^*|$, then we have
$$|X(\tilde{\sigma}_4)|\lesssim \dfrac{1}{\theta_0}m(\xi_4^*)^2.$$
Thus we can assume $|\xi_1|\geq |\xi_2|\geq ...\geq |\xi_6|$. And we
can also assume $|\xi_1|\sim |\xi_2|\gtrsim N$.

\vspace*{2ex}

Case 1. $N_1\sim N_2\gtrsim N, N_3\gtrsim 1$.

\vspace*{2ex}

$\bullet$ Case 1(a) $N_6\gtrsim 1$. Observe that
\begin{align*}
&m(N_4)^2\int_0^T\int_{\Sigma_6} \prod\limits_{j=1}^6\hat{u}_j
dt\\
\lesssim & m(N_4)^2\sup\limits_{|\xi_6|\sim N_6,|\xi_5|\sim
N_5}\Big(\int_0^T\int_{\sum\limits_{j=1}^4\xi_j=-\xi_5-\xi_6}\prod\limits_{j=1}^4\hat{u}_jdt\Big)
\int_{|\xi_6|\sim N_6}\hat{u}_6d\xi_6 \int_{|\xi_5|\sim
N_5}\hat{u}_5d\xi_5\\
\lesssim & m(N_4)^2\sup\limits_{|\xi_6|\sim N_6,|\xi_5|\sim
N_5}\int_0^T\int_{\sum\limits_{j=1}^4\xi_j=-\xi_5-\xi_6}\prod\limits_{j=1}^4\hat{u}_jdt
N_6^{1/2}||\nabla u_6||_{L_t^{\infty}L_x^2}N_5^{1/2}||\nabla
u_5||_{L_t^{\infty}L_x^2}\\
\lesssim & N_5\sup\limits_{|\xi_5|\sim N_5,|\xi_6|\sim
N_6}\int_0^T\int_{\sum\limits_{j=1}^4\xi_j=-\xi_5-\xi_6}\prod\limits_{j=1}^4\hat{u}_jdt.
\end{align*}
We decompose $u_1,u_2$ into linear-nonlinear components, i.e.,
$$u_i=u_i^l+u_i^{nl},~i=1,2$$
In the case of $(u_1^l,u_2^l)$, we have
\begin{align*}
& N_5\int_0^T\int_{\sum\limits_{j=1}^4\xi_j=-\xi_5-\xi_6}\hat{u_1^l}\hat{u_2^l}\hat{u}_3\hat{u}_4dt\\
\lesssim & N_5
||u_1^l||_{L_t^2L_x^6}||u_2^l||_{L_t^2L_x^6}||u_3||_{L_t^{\infty}L_x^2}||u_4||_{L_t^{\infty}L_x^6}\\
\lesssim &
N_5N_1^{-1}N_2^{-1}N_3^{-1}m(N_1)^{-1}m(N_2)^{-1}m(N_3)^{-1}m(N_4)^{-1}\\
\lesssim & N_1^{0-}N^{-2+}.
\end{align*}

If there is one nonlinear term, for example, $(u_1^{l},u_2^{nl})$,
then we obtain
\begin{align*}
& N_5\int_0^T\int_{\sum\limits_{j=1}^4\xi_j=-\xi_5-\xi_6}\hat{u_1^l}\hat{u_2^{nl}}\hat{u}_3\hat{u}_4dt\\
\lesssim & N_5
||u_1^l||_{L_t^2L_x^6}||u_2^{nl}||_{L_t^{\infty}L_x^2}||u_3||_{L_t^{2}L_x^6}||u_4||_{L_t^{\infty}L_x^6}\\
\lesssim &
N_5N_1^{-1}N_2^{-1}N_2^{-1}N_3^{-1}m(N_1)^{-1}m(N_2)^{-1}m(N_3)^{-1}m(N_4)^{-1}M(J,u,2)\\
\lesssim & N_1^{0-}N^{-3+}M(J,u,2).
\end{align*}

If there are two nonlinear terms, then we get
\begin{align*}
N_5\int_0^T\int_{\Sigma_6}
\hat{u_1^{nl}}\hat{u_2^{nl}}\prod\limits_{j=3}^6\hat{u}_j dt
\lesssim &
N_5||u_1^{nl}||_{L_t^{\infty}L_x^2}||u_2^{nl}||_{L_t^{2}L_x^6}||u_3||_{L_t^{2}L_x^{6}}
||u_4||_{L_t^{\infty}L_x^{6}}\\
\lesssim & N_1^{-}N^{-7/2+}M(J,u,1).
\end{align*}

$\bullet$ Case 1(b) $N_6\ll 1$. For this cae, we need a factor
$N_6^+$ to sum over $N_6$. Again we decompose $u_1,u_2$ into
linear-nonlinear components. We can argue exactly as in Case 1(a) to
get
\begin{align*}
& m(N_4)^2\int_0^T\int_{\Sigma_6}
\hat{u_1^{l}}\hat{u_2^{l}}\prod\limits_{j=3}^6\hat{u}_j dt\lesssim
N_1^-N_6^{1/2}N^{-2+};
\\
&m(N_4)^2\int_0^T\int_{\Sigma_6}
\hat{u_1^{nl}}\hat{u_2^{l}}\prod\limits_{j=3}^6\hat{u}_j dt\lesssim
N_1^-N_6^{1/2}N^{-3+}M(J,u,2);
\\
&m(N_4)^2\int_0^T\int_{\Sigma_6}
\hat{u_1^{l}}\hat{u_2^{nl}}\prod\limits_{j=3}^6\hat{u}_j dt\lesssim
N_1^-N_6^{1/2}N^{-3+}M(J,u,2)
\\
&m(N_4)^2\int_0^T\int_{\Sigma_6}
\hat{u_1^{nl}}\hat{u_2^{nl}}\prod\limits_{j=3}^6\hat{u}_j dt\lesssim
N_1^-N_6^{1/2}N^{-7/2+}M(J,u,1).
\end{align*}

 Case 2. $N_1\sim N_2\gtrsim N, N_3\ll 1$. Similar to Case 1(b), we
 decompose $u_1,u_2$ into linear and nonlinear components.

In order to obtain a factor $N_6^+$, we interpolate $||
Iu_i||_{L_t^{\infty}L_x^6}$ and $||Iu_i||_{L_{t,x}^4}$. Note that by
Sobolev embeddding,
$$||Iu_i||_{L_t^{\infty}L_x^6}\lesssim ||\nabla
Iu_i||_{L_t^{\infty}L_x^2}\lesssim 1.$$ Then for $0\leq a\leq 1$,
since
\begin{align*}
\begin{cases}
\frac{1-a}{4}=\frac{a}{\infty}+\frac{1-a}{4},\\
\frac{3-a}{12}=\frac{a}{6}+\frac{1-a}{4},
\end{cases}
\end{align*}
we have
$$
||Iu_i||_{L_t^{\frac{4}{1-a}}L_x^{\frac{12}{3-a}}(J\times\mathbb{R}^3)}\lesssim
||
Iu_i||_{L_t^{\infty}L_x^6}^a||Iu_i||_{L_{t,x}^4(J\times\mathbb{R}^3)}^{1-a}\lesssim
M(J,u,4/(1-a)),\eqno{(**)}
$$
where the last inequality is by the definition of $M(J,u,q)$. Take
$a=1-$, then $M(J,u,4/(1-a))=M(J,u,\infty-)$. Note that
$(\frac{4}{1+a},\frac{6}{2-a})$ is admissible. By Bernstein
inequality (Lemma \ref{bernstein}), inequality $(**)$, we have
\begin{align*}
&m(N_4)^2\int_0^T\int_{\Sigma_6} \hat{u_1^{l}}\hat{u_2^{l}}\prod\limits_{j=3}^6\hat{u}_j dt\\
\lesssim &
N_5^{1/2}N_6^{1/2}\int_0^T\int_{\sum\limits_{j=1}^4\xi_j=-\xi_5-\xi_6}
\hat{u_1^l}\hat{u_2^l}\hat{u_3}\hat{u}_4dt\\
\lesssim & N_5^{1/2}N_6^{1/2}
||u_1^l||_{L_{t}^{\frac{4}{1+a}}L_x^{\frac{6}{2-a}}}||u_2^l||_{L_{t}^{2}L_x^6}||Iu_3||_{L_t^{\infty}L_x^{\frac{4}{1+a}}}
||Iu_4||_{L_t^{\frac{4}{1-a}}L_x^{\frac{12}{3-a}}}\\
\lesssim & N_5^{1/2}N_6^{1/2}N_3^{\frac{3(1-a)}{4}}
||u_1^l||_{L_{t}^{\frac{4}{1+a}}L_x^{\frac{6}{2-a}}}||u_2^l||_{L_{t}^{2}L_x^6}||Iu_3||_{L_t^{\infty}L_x^{2}}
||Iu_4||_{L_t^{\frac{4}{1-a}}L_x^{\frac{12}{3-a}}}\\
\lesssim &
N_5^{1/2}N_6^{1/2}N_1^{-1}N_2^{-1}N_3^{-1}N_3^{\frac{3(1-a)}{4}}m(N_1)^{-1}m(N_2)^{-1}M(J,u,4/(1-a))\\
\lesssim & N_1^-N_6^{\frac{3(1-a)}{4}} N^{-2+}M(J,u,4/(1-a))\\
\lesssim & N_1^- N_6^+ N^{-2+}M(J,u,\infty-),
\end{align*}
\begin{remark}
The presence of $M(J,u,\infty-)$ is not essential. As we can see in
section 6, $M(J,u,\infty-)\sim N^-$, so $M(J,u,\infty-)N^{-1+}\sim
N^{-1++}$. Thus we omit the factor $M(J,u,\infty-)$ throughout this
paper.
\end{remark}

Use similar argument as the above, we obtain
\begin{align*}
&m(N_4)^2\int_0^T\int_{\Sigma_6} \hat{u_1^{nl}}\hat{u_2^{l}}\prod\limits_{j=3}^6\hat{u}_j dt\\
\lesssim & N_5^{1/2}N_6^{1/2}\int_0^T\int_{\sum\limits_{j=1}^4\xi_j=-\xi_5-\xi_6}\hat{u_1^{nl}}\hat{u_2^l}\hat{u_3}\hat{u}_4dt\\
\lesssim & N_5^{1/2}N_6^{1/2}
||u_1^{nl}||_{L_t^{\infty}L_x^{\frac{6}{2+a}}}||u_2^l||_{L_t^{\frac{4}{1+a}}L_x^{\frac{6}{2-a}}}||u_3||_{L_t^{2}L_x^{\frac{12}{1+a}}}||u_4||_{L_t^{\frac{4}{1-a}}L_x^{\frac{12}{3-a}}}\\
\lesssim &
N_5^{1/2}N_6^{1/2}N_1^{(1-a)/2}N_1^{-1}N_2^{-1}N_3^{-1}N_3^{(1-a)/4}m(N_1)^{-1}m(N_2)^{-1}\\
&~~~~~~~~~~~~\times||\nabla Iu_1^{nl}||_{L_t^{\infty}L_x^{2}}
||\nabla Iu_2^{l}||_{L_t^{2+}L_x^{6-}}||\nabla Iu_3||_{L_t^{2}L_x^6}M(J,u,4/(1-a))\\
\lesssim & N_1^{0-}N_6^+N^{-3+}M(J,u,2).
\end{align*}

The
 $(u_1^{nl},u_2^{nl})$ case:
\begin{align*}
&m(N_4)^2\int_0^T\int_{\Sigma_6} \hat{u_1^{nl}}\hat{u_2^{nl}}\prod\limits_{j=3}^6\hat{u}_j dt\\
\lesssim & N_5^{1/2}N_6^{1/2}\int_0^T\int_{\sum\limits_{j=1}^4\xi_j=-\xi_5-\xi_6}\hat{u_1^{nl}}\hat{u_2^{nl}}\hat{u_3}\hat{u}_4dt\\
\lesssim & N_5^{1/2}N_6^{1/2}
||u_1^{nl}||_{L_t^{\infty}L_x^{\frac{6}{2+a}}}||u_2^{nl}||_{L_t^{\frac{4}{1+a}}L_x^{\frac{6}{2-a}}}||u_3||_{L_t^{2}L_x^{\frac{12}{1+a}}}||u_4||_{L_t^{\frac{4}{1-a}}L_x^{\frac{12}{3-a}}}\\
\lesssim &
N_5^{1/2}N_6^{1/2}N_1^{(1-a)/2}N_1^{-1}N_2^{-1}N_3^{-1}N_3^{(1-a)/4}m(N_1)^{-1}m(N_2)^{-1}\\
&~~~~~~~~~~~~\times||\nabla Iu_1^{nl}||_{L_t^{\infty}L_x^{2}}
||\nabla Iu_2^{nl}||_{L_t^{2+}L_x^{6-}}||\nabla Iu_3||_{L_t^{2}L_x^6}M(J,u,4/(1-a))\\
\lesssim & N_1^{0-}N_6^+N^{-7/2+}M(J,u,1).
\end{align*}

This ends the proof of Lemma \ref{sexti}.
\end{proof}

\subsubsection{Quadrilinear Estimate}

We prove the quarilinear estimate. We first show the following
lemma.
\begin{lemma}\label{quadril}
Let $u(x,t)$ be a smooth in time, schwartz in space solution to
(\ref{equ}) with initial data $u_0\in H_x^{s}(\mathbb{R}^3)(s>1/2)$
 defined on $J\times\mathbb{R}^3$ such that
\begin{equation}
\sup\limits_{t\in J} E(Iu(t))\leq
1,~||u||_{L_{t,x}^4(J\times\mathbb{R}^3)}< \infty,
\end{equation}
then
\begin{equation}\label{quad-sch}
\begin{aligned}
&|\int_0^{t_0}
\Lambda_4([-2iX(\sigma_2)]_{\text{sym}}+i\tilde{\sigma}_4\alpha;u(t))dt|\\
\lesssim &
 \max\{\dfrac{\theta_0}{N^{1/2-}},N^{-3/2+}, \dfrac{M(J,u,2)}{N^{5/2-}},
 \dfrac{M(J,u,1)}{N^{13/4-}},
\theta_0\dfrac{M(J,u,2)}{N^{7/4-}},\theta_0\dfrac{M(J,u,1)}{N^{9/4-}}\}.
\end{aligned}
\end{equation}
\end{lemma}
\begin{proof}

From (\ref{tilde}) we have
$$([-2iX(\sigma_2)]_{\text{sym}}+i\tilde{\sigma}_4\alpha_4)(\xi)=[-2iX(\sigma_2)]_{\text{sym}}1_{\Omega_{\text{res}}}=\dfrac{i}{4}\sum\limits_{j=1}^4(-1)^{j+1}
m_j^2|\xi_j|^21_{\Omega_{\text{res}}},
$$
where the resonant set
$$\Omega_{\text{res}}:=\{(\xi_1,\xi_2,\xi_3,\xi_4)\in \Sigma_4\mid
\max\limits_{1\leq i\leq
4}\{|\xi_i|\}>N;|\text{cos}\angle(\xi_{12},\xi_{14})|<\theta_0\}.$$

As in the above, we decompose $u_i(i=1,2,3,4)$ into dyadic pieces
such that $|\xi_i|\sim N_i$. By symmetry, we may assume that
$N_1\geq N_2,N_3,N_4$, and $N_2\geq N_4$. Thus we can further assume
$N_2\geq N_3\geq N_4$ by symmetry argument. Denote
\begin{align*}
\Omega_r=\Big\{(\xi_1,\xi_2,\xi_3,\xi_4)\in \Sigma_4\mid N_1>N;
N_1\sim N_2; N_1\geq N_2\geq & N_3\geq
N_4,\\
&|\text{cos}\angle(\xi_{12},\xi_{14})|<\theta_0\Big\}.
\end{align*}

Then it suffices to show
\begin{equation}\label{cases}
\begin{aligned}
&\int_0^T\int_{\Omega_{r}}\Big(\sum\limits_{j=1}^4(-1)^{j+1}m(\xi_1)^2|\xi_1|^2\Big)\hat{u}(\xi_1)\hat{u}(\xi_2)\hat{u}(\xi_3)\hat{u}(\xi_4)\\
\lesssim &
 \max\{\dfrac{\theta_0}{N^{1/2-}},N^{-3/2+}, \dfrac{M(J,u,2)}{N^{5/2-}},
 \dfrac{M(J,u,1)}{N^{13/4-}},
\theta_0\dfrac{M(J,u,2)}{N^{7/4-}},\theta_0\dfrac{M(J,u,1)}{N^{9/4-}}\}.
\end{aligned}
\end{equation}

 \noindent
Observe that on $\Omega_r$,
$$|\xi_1|^2-|\xi_2|^2+|\xi_3|^2-|\xi_4|^2=2|\xi_{12}||\xi_{14}||\text{cos}\angle(\xi_{12},\xi_{14})|\lesssim |\xi_{12}||\xi_{14}|\theta_0.$$
Also note that
$$|\xi_1|^2-|\xi_2|^2=(|\xi_1|+|\xi_2|)(|\xi_1|-|\xi_2|)\geq |\xi_{1}+\xi_2|(|\xi_1|-|\xi_2|)=|\xi_{12}|(|\xi_1|-|\xi_2|)$$
and
$$|\xi_3|^2-|\xi_4|^2=(|\xi_3|+|\xi_4|)(|\xi_3|-|\xi_4|)\geq |\xi_{3}+\xi_4|(|\xi_3|-|\xi_4|)=|\xi_{12}|(|\xi_3|-|\xi_4|).$$
Thus we have
\begin{equation}
|\xi_1|-|\xi_2|\lesssim |\xi_1|\theta_0,~~ |\xi_3|-|\xi_4|\lesssim
|\xi_1|\theta_0.
\end{equation}

To finish the proof of (\ref{cases}), we consider four cases.

Case I. $N_1\geq N_2\geq N_3\geq N_4\gtrsim N$. Then we have
\begin{align*}
\sum\limits_{j=1}^4(-1)^{j+1}m(\xi_1)^2|\xi_1|^2\lesssim &
\dfrac{N^{2-2s}}{|\xi_1|^{2-2s}}|\xi_1|^2-\dfrac{N^{2-2s}}{|\xi_2|^{2-2s}}|\xi_2|^2+\dfrac{N^{2-2s}}{|\xi_3|^{2-2s}}|\xi_3|^2
-\dfrac{N^{2-2s}}{|\xi_4|^{2-2s}}|\xi_4|^2\\
\lesssim &
N^{2-2s}\Big[(|\xi_1|^{2s}-|\xi_2|^{2s})+(|\xi_3|^{2s}-|\xi_4|^{2s})\Big]\\
\lesssim &
N^{2-2s}(|\xi_1|^{2s-1}(|\xi_1|-|\xi_2|)+|\xi_3|^{2s-1}(|\xi_3|-|\xi_4|)\\
\lesssim &
N^{2-2s}(|\xi_1|^{2s-1}|\xi_1|\theta_0+|\xi_3|^{2s-1}|\xi_1|\theta_0)\\
\lesssim & N^{2-2s}N_1^{2s}\theta_0.
\end{align*}

\vspace*{2ex}

We decompose $u_1,u_2,u_3$ and obtain
\begin{align*}
&\int_0^T\int_{\Omega_r}\Big(\sum\limits_{j=1}^4(-1)^{j+1}m(\xi_i)^2|\xi_i|^2\Big)\hat{u_1^l}(\xi_1)\hat{u_2^l}(\xi_2)\hat{u_3^l}(\xi_3)\hat{u_4}(\xi_4)\\
\lesssim & N^{2-2s}N_1^{2s}\theta_0\int_0^T \int_{\Sigma_4}\hat{u_1^l}(\xi_1)\hat{\bar{u_2^l}}(\xi_2)\hat{u_3^l}(\xi_3)\hat{u_4}(\xi_4)\\
\lesssim & N^{2-2s}N_1^{2s}\theta_0
||u_1^l||_{L_t^{\infty}L_x^2}||u_2^l||_{L_t^2L_x^6}||u_3^l||_{L_t^2L_x^6}||u_4||_{L_t^{\infty}L_x^6}\\
\lesssim &N_1^-N^{-1+}\theta_0.
\end{align*}
Next if there is one nonlinear term, for example, $(u_1^{nl},u_2^l,
u_3^l),$ then
\begin{align*}
&\int_0^T\int_{\Omega_r}\Big(\sum\limits_{j=1}^4(-1)^{j+1}m(\xi_i)^2|\xi_i|^2\Big)\hat{u_1^{nl}}(\xi_1)\hat{u_2^l}(\xi_2)\hat{u_3^l}(\xi_3)\hat{u}_4(\xi_4)\\
\lesssim & N^{2-2s}N_1^{2s}\theta_0\int_0^T \int_{\Sigma_4}\hat{u_1^{nl}}(\xi_1)\hat{u_2^l}(\xi_2)\hat{u_3^l}(\xi_3)\hat{u}_4(\xi_4)\\
\lesssim & N^{2-2s}N_1^{2s}\theta_0
||u_1^{nl}||_{L_t^{\infty}L_x^2}||u_2^l||_{L_t^2L_x^6}||u_3^l||_{L_t^2L_x^6}||u_4||_{L_t^{\infty}L_x^6}\\
\lesssim &N_1^-N^{-2+}\theta_0.
\end{align*}

If there are two nonlinear terms, for example, $(u_1^{nl},u_2^{nl},
u_3^l)$, take $L_t^{\infty}L_x^2$, $L_t^2L_x^6$, $L_t^2L_x^6$,
$L_t^{\infty}L_x^6$ for $u_1, u_2, u_3, u_4$, respectively, then the
above argument implies that
\begin{align*}
&\int_0^T\int_{\Omega_r}\Big(\sum\limits_{j=1}^4(-1)^{j+1}m(\xi_i)^2|\xi_i|^2\Big)\hat{u_1^{nl}}(\xi_1)\hat{u_2^{nl}}(\xi_2)\hat{u_3^l}(\xi_3)\hat{u_4}(\xi_4)
\\
\lesssim & N_1^-N^{-5/2+}\theta_0M(J,u,2).
\end{align*}
If there are three nonlinear terms, say, $(u_1^{nl},u_2^{nl},
u_3^{nl})$, take $L_t^{\infty}L_x^2$, $L_t^2L_x^6$, $L_t^2L_x^6$,
$L_t^{\infty}L_x^6$ for $u_1, u_2, u_3, u_4$, respectively, then we
have
\begin{align*}
&\int_0^T\int_{\Omega_r}\Big(\sum\limits_{j=1}^4(-1)^{j+1}m(\xi_i)^2|\xi_i|^2\Big)\hat{u_1^{nl}}(\xi_1)\hat{u_2^{nl}}(\xi_2)\hat{u_3^{nl}}(\xi_3)\hat{u_4}(\xi_4)
\\
\lesssim & N_1^-N^{-3+}\theta_0M(J,u,1).
\end{align*}

 \vspace*{2ex}

Case II. $N_3\gtrsim N$, $1\lesssim N_4\ll N.$ For this case we have
\begin{align*}
\sum\limits_{j=1}^4(-1)^{j+1}m(\xi_1)^2|\xi_1|^2\lesssim &
\dfrac{N^{2-2s}}{|\xi_1|^{2-2s}}|\xi_1|^2-\dfrac{N^{2-2s}}{|\xi_2|^{2-2s}}|\xi_2|^2+\dfrac{N^{2-2s}}{|\xi_3|^{2-2s}}|\xi_3|^2
-|\xi_4|^2\\
\lesssim &
N^{2-2s}(|\xi_1|^{2s}-|\xi_2|^{2s})+(|\xi_3|^{2}-|\xi_4|^{2})\\
\lesssim &
N^{2-2s}(|\xi_1|^{2s-1}(|\xi_1|-|\xi_2|))+|\xi_1||\xi_3|\theta_0\\
\lesssim & N^{2-2s}N_1^{2s}\theta_0+N_1N_3\theta_0.
\end{align*}

By the same argument as in Case I, we obtain
\begin{align*}
&\int_0^T\int_{\Omega_r}\Big(\sum\limits_{j=1}^4(-1)^{j+1}m(\xi_1)^2|\xi_1|^2\Big)\hat{u_1}(\xi_1)\hat{u_2}(\xi_2)\hat{u_3}(\xi_3)\hat{u_4}(\xi_4)\\
\lesssim &
N_1^-\theta_0(N^{-3/2+}+N^{-5/2+}M(J,u,2)+N^{-3+}M(J,u,1)).
\end{align*}

Case III. $N_3\gtrsim N$, $N_4\ll 1$. The argument is similar to
Case I and Case II except that we can obtain an $N_4^+$ factor to
sum over $N_4$ directly. More precisely,
\begin{align*}
&\int_0^T\int_{\Omega_r}\Big(\sum\limits_{j=1}^4(-1)^{j+1}m(\xi_1)^2|\xi_1|^2\Big)\hat{u_1}(\xi_1)\hat{u_2}(\xi_2)\hat{u_3}(\xi_3)\hat{u_4}(\xi_4)\\
\lesssim &
N_1^-N_4^+\theta_0(N^{-3/2}+N^{-5/2+}M(J,u,2)+N^{-3+}M(J,u,1)).
\end{align*}
\vspace*{2ex}

Case IV. $N_4\leq N_3\ll N$. For this case we need the following
lemma in \cite{CKSTT4}. The reader may refer to \cite{CKSTT4} for
the proof.
\begin{lemma}
Let $N_1\geq N_2\geq N_3\geq N_4$,
$N_1\sim N_2\gtrsim N$, $N_3\ll N$. Let
$(\xi_1,\xi_2,\xi_3,\xi_4)\in \Omega_{r}$ be such that $|\xi_j|\sim
N_j(j=1,2,3,4)$. Then
\begin{equation}
|m^2(\xi_1)|\xi_1|^2-m^2(\xi_2)|\xi_2|^2+m^2(\xi_3)|\xi_3|^2-m^2(\xi_4)|\xi_4|^2|\lesssim
m(N_1)^2N_1N_3\theta_0+m(N_3)^2N_3^2.
\end{equation}
\end{lemma}

Case IV is divided into three subcases.

\vspace*{1ex} Case IV(a). $N_3 \ll 1$. We argue similar to Case 2 of
Lemma \ref{sexti}. We decompose $u_1$ and $u_2$ into linear and
nonlinear parts. Again we use the estimate
$$
||Iu_i||_{L_t^{\frac{4}{1-a}}L_x^{\frac{12}{3-a}}(J\times\mathbb{R}^3)}\lesssim
||
Iu_i||_{L_t^{\infty}L_x^6}^a||Iu_i||_{L_{t,x}^4(J\times\mathbb{R}^3)}^{1-a}\lesssim
M(J,u,4/(1-a)),\eqno{(**)}
$$
Let $a=1-$, for $(u_1^l,u_2^l)$, we have
\begin{align*}
&m(N_1)^2N_1N_3\theta_0 \int_0^T\int_{\Omega_r}\hat{u_1^l}(\xi_1)\hat{u_2^l}(\xi_2)\hat{u_3}(\xi_3)\hat{u_4}(\xi_4)\\
\lesssim &
m(N_1)^2N_1N_3\theta_0||u_1^l||_{L_{t}^{\frac{4}{1+a}}L_x^{\frac{6}{2-a}}}||u_2^l||_{L_{t}^{2}L_x^6}
||Iu_3||_{L_t^{\infty}L_x^{\frac{24}{7+5a}}}||Iu_4||_{L_t^{\frac{4}{1-a}}L_x^{\frac{24}{5-a}}}\\
 \lesssim &\theta_0N_3^{\frac{5(1-a)}{8}}N_4^{\frac{1-a}{8}}M(J,u,4/(1-a))
||\nabla
Iu_1^l||_{L_{t}^{\frac{4}{1+a}}L_x^{\frac{6}{2-a}}}||Iu_2^l||_{L_{t}^{2}L_x^6}
||\nabla Iu_3||_{L_t^{\infty}L_x^2}\\
&~~~~~~~~~~~~~~~~~~~~~~~~~~~~~~~~~~~~~~~~~~~~~~~~~~~~~~~~~~~\times ||Iu_4||_{L_t^{\frac{4}{1-a}}L_x^{\frac{12}{3-a}}}\\
\lesssim & N_1^-N^{-1+}N_4^+\theta_0.
\end{align*}
If only one nonlinear term appears, then we argue similarly. For
example,
\begin{align*}
&m(N_1)^2N_1N_3\theta_0 \int_0^T\int_{\Omega_r}\hat{u_1^{nl}}(\xi_1)\hat{u_2^l}(\xi_2)\hat{u_3}(\xi_3)\hat{u_4}(\xi_4)\\
\lesssim &
m(N_1)^2N_1N_3\theta_0||u_1^{nl}||_{L_t^{\infty}L_x^{\frac{6}{2+a}}}||u_2^l||_{L_t^{\frac{4}{1+a}}L_x^{\frac{6}{2-a}}}||u_3||_{L_t^{2}L_x^{\frac{24}{3+a}}}
||u_4||_{L_t^{\frac{4}{1-a}}L_x^{\frac{24}{5-a}}}\\
\lesssim &\theta_0N_1^+N_3^+N_4^+M(J,u,\infty-) ||\nabla
Iu_1^{nl}||_{L_{t}^{\infty}L_x^{2}}||Iu_2^l||_{L_{t}^{2+}L_x^{6-}}
||\nabla
Iu_3||_{L_t^{2}L_x^6}||Iu_4||_{L_t^{\frac{4}{1-a}}L_x^{\frac{12}{3-a}}}\\
\lesssim & N_1^-N^{-2+}N_4^+\theta_0M(J,u,2).
\end{align*}

If two nonlinear terms appear, then
\begin{align*}
&m(N_1)^2N_1N_3\theta_0 \int_0^T\int_{\Omega_r}\hat{u_1^{nl}}(\xi_1)\hat{u_2^{nl}}(\xi_2)\hat{u_3}(\xi_3)\hat{u_4}(\xi_4)\\
\lesssim &
m(N_1)^2N_1N_3\theta_0||u_1^{nl}||_{L_t^{\infty}L_x^{\frac{6}{2+a}}}||u_2^l||_{L_t^{\frac{4}{1+a}}L_x^{\frac{6}{2-a}}}||u_3||_{L_t^{2}L_x^{\frac{24}{3+a}}}
||u_4||_{L_t^{\frac{4}{1-a}}L_x^{\frac{24}{5-a}}}\\
\lesssim &\theta_0N_1^+N_3^+N_4^+M(J,u,\infty-) ||\nabla
Iu_1^{nl}||_{L_{t}^{\infty}L_x^{2}}||Iu_2^l||_{L_{t}^{2+}L_x^{6-}}
||\nabla
Iu_3||_{L_t^{2}L_x^6}||Iu_4||_{L_t^{\frac{4}{1-a}}L_x^{\frac{12}{3-a}}}\\
\lesssim & N_1^-N^{-5/2+}N_4^+\theta_0M(J,u,1).
\end{align*}

 Similarly, we have
\begin{align*}
&m(N_3)^2N_3^2\int_0^T\int_{\Omega_r}\hat{u_1}(\xi_1)\hat{u_2}(\xi_2)\hat{u_3}(\xi_3)\hat{u_4}(\xi_4)\\
\lesssim & N_1^-N_4^+(N^{-2+}+N^{-3+}M(J,u,2)+N^{-7/2+}M(J,u,1)).
\end{align*}

 Case IV(b). $N_4\ll 1, N_3\gtrsim 1$.

\vspace*{1ex}

 $\bullet$ If $1\lesssim N_3\ll N^{1/2}$

 First estimate
\begin{align*}
I_2:=m(N_1)^2N_1N_3\theta_0
\int_0^T\int_{\Omega_r}\hat{u_1}(\xi_1)\hat{u_2}(\xi_2)\hat{u_3}(\xi_3)\hat{u_4}(\xi_4).
\end{align*}
 Again we decompose $u_1,u_2$ into linear and nonlinear components.
 The cases $(u_1^l,u_2^l)$, $(u_1^{nl},
 u_2^{l})$, $(u_1^l,u_2^{nl})$ are
 easy to deal with. For example, we have
\begin{align*}
&m(N_3)^2N_1N_3\theta_0\int_0^T\int_{\Omega_r}\hat{u_1^l}(\xi_1)\hat{u_2^{nl}}(\xi_2)\hat{u_3}(\xi_3)\hat{u_4}(\xi_4)\\
\lesssim &
m(N_3)^2N_1N_3\theta_0||u_1^l||_{L_t^{2+}L_x^{6-}}||u_2^{nl}||_{L_t^{\infty}L_x^2}||u_3||_{L_t^2L_x^6}||u_4||_{L_t^{\infty-}L_x^{6+}}\\
\lesssim & N_1^-N_4^+N^{-2+}\theta_0M(J,u,2).
\end{align*}
It remains to deal with the case $(u_1^{nl},u_2^{nl})$. We have
\begin{align*}
&m(N_3)^2N_1N_3\theta_0\int_0^T\int_{\Omega_r}\hat{u_1^{nl}}(\xi_1)\hat{u_2^{nl}}(\xi_2)\hat{u_3}(\xi_3)\hat{u_4}(\xi_4)\\
\lesssim &
m(N_3)^2N_1N_3\theta_0||u_1^{nl}||_{L_t^{\infty}L_x^{2}}||u_2^{nl}||_{L_t^{2+}L_x^{6-}}||u_3||_{L_t^{2}L_x^6}||u_4||_{L_t^{\infty-}L_x^{6+}}\\
\lesssim & N_1^-N_4^+N^{-5/2+}\theta_0M(J,u,1).
\end{align*}

Next, since $1\leq N_3\ll N^{1/2}$, by decomposing $u_1, u_2$ and
$u_3$, we get
\begin{align*}
&m(N_3)^2N_3^2\int_0^T\int_{\Omega_r}\hat{u_1}(\xi_1)\hat{u_2}(\xi_2)\hat{u_3}(\xi_3)\hat{u_4}(\xi_4)\\
\lesssim &
N_1^-N_4^+(N^{-3/2+}+N^{-11/4+}M(J,u,2)+N^{-13/4+}M(J,u,1)).
\end{align*}

$\bullet$ If $N_3\gtrsim N^{1/2}$.

We use the bound
\begin{align*}
\sum\limits_{j=1}^4(-1)^{j+1}m(\xi_1)^2|\xi_1|^2\lesssim
N^{2-2s}N_1^{2s}\theta_0+N_1N_3\theta_0.
\end{align*}

The argument in Case I indeed gives that
\begin{align*}
&(N^{2-2s}N_1^{2s}\theta_0+N_1N_3\theta_0)\int_0^T\int_{\Omega_r}\hat{u_1}(\xi_1)\hat{u_2}(\xi_2)\hat{u_3}(\xi_3)\hat{u_4}(\xi_4)\\
\lesssim &
N_1^-N_4^+\theta_0(N^{-1/2+}+N^{-7/4+}M(J,u,2)+N^{-9/4+}M(J,u,1)).
\end{align*}

Case IV(c). $N_4\gtrsim 1$, $1\lesssim N_3\ll N$. Just argue
similarly.

This ends the proof of Lemma \ref{quadril}.
\end{proof}

\begin{proof}[Proof
of Theorem \ref{modified}] Take $\theta_0=N^{-7/8}$, then Theorem
\ref{modified} follows from Proposition \ref{point} and Proposition
\ref{conservation}.
\end{proof}

\section{Global well-posedness and scattering}
We prove Theorem \ref{main}.
\begin{proof}[Proof of Theorem \ref{main}]
 Choose $\lambda\sim
N^{\frac{1-s}{s-1/2}}$ such that $E(Iu_0^{(\lambda)})\leq 1/4$.
Define
\begin{equation}\label{iteration}
W:=\{T\in [0,\infty): \sup\limits_{0\leq t\leq
T}E(Iu^{(\lambda)}(t))\leq 1/2\}.
\end{equation}
Then $W\neq \emptyset$ since $0\in  W$. Also $W$ is closed by
dominated convergence theorem. Note that if $T\in W$, then we obtain
\begin{align*}
&||u^{(\lambda)}||_{L_{t,x}^4([0,T]\times\mathbb{R}^3)}\\
\leq &
C(||u_0||_{L_x^2})\Big(\lambda^{3/8}\sup\limits_{0\leq t\leq T}
||\nabla
Iu^{(\lambda)}(t)||_{L_x^2}^{1/4}+\lambda^{1/4}\sup\limits_{0\leq
t\leq T} ||\nabla Iu^{(\lambda)}(t)||_{L_x^2}^{1/{4s}}\Big)\\
\leq &
C(||u_0||_{L_x^2})\Big(\dfrac{1}{2}\lambda^{3/8}+\dfrac{1}{2}\lambda^{1/4}\Big)
\\
\leq & C(||u_0||_{L_x^2})\lambda^{3/8}.
\end{align*}
Thus $||u^{(\lambda)}||_{L_{t,x}^4([0,T]\times\mathbb{R}^3)}$ is
uniformly bounded for any $T\in W$.

 We show that $W$ is open so that $W=[0,\infty)$.
Assume $T\in W$. By continuity, there exists $\delta>0$ such that
for each $T'\in (T-\delta,T+\delta)\cap [0,\infty)$,
$$\sup\limits_{t\in [0,T']} E(Iu^{(\lambda)}(t))\leq 1,~~~||u^{(\lambda)}||_{L_{t,x}^4([0,T']\times\mathbb{R}^3)}\leq 2C(||u_0||_{L_x^2})\lambda^{3/8}.$$

Now we decompose $[0,T']$ into $\lambda^{27/50}$ subintervals
$\{J_m\}_{m=1}^{\lambda^{27/50}}$ such that for each $J_m$,
$$||u^{(\lambda)}||_{L_{t,x}^4(J_m\times\mathbb{R}^3)}^4\lesssim \lambda^{24/25}. $$

Note that $\lambda^{24/25}\leq N^{2}$ provided $s\geq 49/74$. Thus
if we choose $s>49/74$, then we have
$$\max\{1,\dfrac{\max\{1,\lambda^{12/25}\}}{N^{1-}},
\dfrac{\max\{1,\lambda^{24/25}\}}{N^{2-}}\}\lesssim 1.
$$
Thus we can choose $N$ so large such that
$$\sup\limits_{t\in [0,T']} |\tilde{E}(Iu^{(\lambda)}(t))-\tilde{E}(Iu^{(\lambda)}(0))|\leq 1/8.$$
By choosing $N$ large enough, we obtain
\begin{align*}
&|E(Iu(t))-E(Iu(0))|\\
\leq &
|E(Iu(t))-\tilde{E}(u(t))|+|\tilde{E}(u(t))-\tilde{E}(u(0))|+|\tilde{E}(u(0))-E(u(0))|\\
\lesssim & N^{-1/8+}+1/8\lesssim 1/4.
\end{align*}
 Thus
$$\sup\limits_{t\in [0,T']} E(Iu^{(\lambda)}(t))\leq 1/2.$$
Hence $T'\in W$. So $W$ is open, which implies that $W=[0,\infty)$.

Scattering follows from standard argument.
\end{proof}

\vskip 0.5cm

\noindent \textbf{Acknowledgement}. The author is indebted to the
referee for many invaluable suggestions, in particular, for pointing
out that interpolating $E(Iu(t)$ with $||u(t)||_{L_{t,x}^4}$ could
obtain a gain to control low frequencies, and encouraging the author
to improve the result to $s>49/74$. The author would like to thank
B. Dodson for sending his paper to the author, from which the author
benefits a lot. The author is also grateful to Han Yongsheng's help
and encouragement.

\end{document}